\documentclass{article}
\usepackage{graphicx,color,amsfonts,amsmath,mathrsfs,hyperref}          

\newcommand{\R}{\mathbb{R}}

\newtheorem{theorem}{Theorem}
\newtheorem{lemma}{Lemma}
\newtheorem{corollary}[lemma]{Corollary}
\newtheorem{proposition}[lemma]{Proposition}
\newtheorem{definition}[lemma]{Definition}

\newtheorem{remark}[lemma]{Remark}

\newenvironment{proof}{{\it Proof.}~~}{\hfill$\square$}

\newcommand{\D}{\mathcal{D}}

\begin{document}


\title{Consensus and Flocking under Communication Failure}


\author{Fabio Ancona\footnote{Dipartimento di Matematica ``Tullio Levi-Civita'', Universit\`a\ degli Studi di Padova, Via Trieste 63, 35121 Padova, Italy, {\tt ancona@math.unipd.it}}, Mohamed Bentaibi\footnote{Dipartimento di Matematica ``Tullio Levi-Civita'', Universit\`a\ degli Studi di Padova, Via Trieste 63, 35121 Padova, Italy, {\tt mohamed.bentaibi@studenti.unipd.it}}, Francesco Rossi\footnote{Dipartimento di Culture del Progetto, Universit\`a\ Iuav di Venezia, 30135 Venezia, Italy, {\tt francesco.rossi@iuav.it}}}


\maketitle


\begin{abstract}                          
For networked systems, Persistent Excitation and Integral Scrambling Condition are conditions ensuring that communication failures between agents can occur, but a minimal level of service is ensured. We consider cooperative multi-agent systems satisfying either of such conditions. For first-order systems, we prove that consensus is attained. For second-order systems, flocking is attained under a standard condition of nonintegrability of the interaction function. In both cases and under both conditions, the original goal is reached under no additional hypotheses on the system with respect to the case of no communication failures.
 \end{abstract}


In recent years, the study of multi-agent systems  has drawn a huge interest in the control community. The reasons behind the impressive rise of this research topic lie in the simplicity of definition, robustness of results, extension of possible applications. For general books on this topic, see e.g. \cite{bullo2009distributed,nedich2015convergence,mesbahi2010graph}. 

Among multi-agent systems, the setting of cooperative systems plays a central role. Very generally, first-order cooperative systems are of the following form:
\begin{eqnarray}\label{e-ODE-noM}
    \dot{x}_i(t) = \frac{\lambda_i}{N}\sum_{j=1}^{N} \phi_{ij}(t) (x_j(t) - x_i(t))
\end{eqnarray}
for $i \in \{1,\ldots,N\}$ where
\begin{eqnarray}
    \phi_{ij}(t) = \phi(|x_i(t) - x_j(t)|)\geq 0. \label{e-phi}
\end{eqnarray}
This dynamics describes the evolution of $N \geq 2$ agents on an Euclidean space $\R^d$, where the position $x_i(t) \in \mathbb{R}^d$ may represent opinion on different topics, velocity or other attributes of agent $i$ at time $t$. The (nonlinear) influence function $\phi_{ij}(t):\R\to\R$ is used to quantify the influence of agent $j$ on agent $i$, where $i,j\in\{1,\ldots,N\}$. The term $\lambda_i$ is a positive scaling parameter. The positivity of $\phi$ corresponds to cooperativity, i.e. to the fact that influence of agents tends to bring them closer. Models of this kind are now ubiquitous in the description of networked systems and their applications. One of the most influential is certainly the bounded confidence model for opinion formation, first described by Hegselmann and Krause in \cite{HK}. The natural goal for such kind of systems is to reach a consensus, i.e. that there exists a common $x^*$ such that $\lim_{t\to +\infty} x_i(t)=x^*$ for all agents $i \in \{1,\ldots,N\}$.

This idea of cooperation has also been extended to second-order systems, in which each agent is described by a pair $(x_i,v_i)$ of position-velocity variables. The dynamics is then given by 
\begin{eqnarray}\label{e-ODE-2nd-noM}
\begin{cases}
    \dot{x}_i(t) = v_i\\
\dot v_i=    \frac{\lambda_i}{N}\sum_{j=1}^{N} \phi_{ij}(t) (v_j(t) - v_i(t))
\end{cases} 
\end{eqnarray}
for $ i,j \in\{1,\ldots,N\}$. As in \eqref{e-phi}, the interaction function $\phi_{ij}$ is non-negative and depends on the positions $x_i,x_j$, but here it plays the role of multiplicative term for the velocity variable. The most famous example of this kind of dynamics was presented by Cucker and Smale in \cite{CS}. A first natural goal for this dynamics is alignment: the velocity variables converge to a common value $v^*$ (i.e. $\lim_{t\to +\infty} v_i(t)=x^*$ for all agents $i \in \{1,\ldots,N\})$. A second goal, that is the focus of this paper, is the so-called flocking: one has both alignment and boundedness of the distances for all times (i.e. there exists $K>0$ such that $|x_i(t)-x_j(t)|< K$ for all $t\geq 0$ and $i,j \in \{1,\ldots,N\}$).\\

From the modelling point of view, each agent is expected to communicate with its neighbours through a \textit{network topology}, influenced by sensor characteristics and the environment. While the easiest scenario involves a fixed network topology (e.g. \cite{watts1998collective,olfati2007consensus}), practical situations often involve dynamic changes, due to factors like communication dropouts, security concerns, or intermittent actuation. In this setting, potential connection losses between agents occur, hindering reaching consensus. Therefore, when interactions between agents are subject to failure, it becomes crucial to investigate whether consensus can still be achieved or not.  For first-order systems, we model this scenario as follows:
\begin{eqnarray}\label{e-ODE}
    \dot{x}_i(t) = \frac{\lambda_i}{N}\sum_{j=1}^{N} M_{ij}(t)\phi_{ij}(t)(x_j(t) - x_i(t))
\end{eqnarray}
for $i \in \{1,\ldots,N\}$. The terms $M_{ij}:[0,+\infty)\mapsto [0,1]$ represent the weight given to the (directed) connection of agent $j$ with agent $i$. They encode the time-varying network topology and account for potential communication failures (e.g., when they vanish). The corresponding second-order system is clearly given by:
\begin{eqnarray}\label{e-ODE-2nd}
\begin{cases}
    \dot{x}_i(t) = v_i\\
\dot v_i=    \frac{\lambda_i}{N}\sum_{j=1}^{N} M_{ij}(t)\phi_{ij}(t) (v_j(t) - v_i(t))
\end{cases} 
\end{eqnarray}

We quantify the possible lack of interactions by introducing two possible conditions: the first one is \textit{persistent excitation} (PE from now on).
\begin{definition}[Persistent excitation]\label{defof:PE} Let $T,\mu>0$ be given. We say that the function $M \in L^{\infty}([0,+\infty);[0,1])$ satisfies the PE condition with parameters $\mu,T$ if it holds
     \begin{equation}\label{e-PE}\tag{PE}
         \int_t^{t+T}M(s)\,ds \geq \mu \qquad \forall t \geq 0.
     \end{equation}
\end{definition}
~

Imposing the PE condition for all $M_{ij}$ means that such a function is not too weak on any given time interval of length $T$. This can be seen as a condition on the minimum level of service.

Although the PE condition is a standard tool in classical control theory (see \cite{narendra2012stable,ChSi2010,ChSi2014}), its use in multi-agent systems has gained interest in the last years (see e.g. \cite{ren2008distributed,tang2020bearing,manfredi2016criterion,bonnet2021consensus}). In \cite{anderson2016convergence}, for instance, the authors prove that consensus holds by requiring the PE condition for all $M_{ij}$ and imposing symmetric communication weights ($M_{ij}=M_{ji}$), with additional technical conditions. Many of the results require the PE condition either on functions that depend on the system's state (see, e.g., \cite{manfredi2016criterion}), or on quantities such as the averaged graph-Laplacian derived from communication weights with respect to the variance bilinear form (see, e.g. \cite{bonnet2021consensus}).

The second condition we present is weaker than the PE condition for all $M_{ij}$: it is the Integral Scrambling Coefficient condition (ISC from now on). It focuses on the family of signals $M_{ij}$ as a whole rather than on each of them.
\begin{definition}[Integral Scrambling Coefficient]\label{defof:ISC} Let $T,\mu>0$ be given. We say that the family of functions $M_{ij} \in L^{\infty}([0,+\infty);[0,1])$ satisfies the ISC condition with parameters $\mu,T$ if for all $i,j\in\{1,\ldots,N\}$ there exists $k\in\{1,\ldots,N\}$ such that 
     \begin{equation}\label{e-ISC}\tag{ISC}
         \min\left\{\int_t^{t+T}M_{ik}(s)\,ds,\int_t^{t+T}M_{jk}(s)\,ds\right\} \geq \mu \quad \forall t \geq 0.
     \end{equation}
\end{definition}
The interpretation of this constraint is the following: for each pair of agents $i,j$, there exists at least a third agent $k$ that is sufficiently often connected to $i$ and to $j$, not necessarily at the same time. One can easily see that ISC is a weaker condition than requiring PE for all $M_{ij}$. Yet, we keep both conditions all along the article, since they are both of interest in themselves.

This condition has also been studied in detail in the literature of networked systems, e.g. in \cite{bonnet2023consensus}. There, the strongest available result is presented: under ISC, the first-order system \eqref{e-ODE} exponentially converges to consensus. Yet, they cannot directly compute the rate of convergence, based on parameters $T,\mu$ or on the initial datum. We later see that this lack of explicit rate of convergence does not allow to extend results to second-order systems. A stronger condition where the scrambling coefficient of the family of functions $M_{ij}$, different from \eqref{e-ISC}, satisfies \eqref{e-PE}, has been considered in \cite{bonnet2022consensus}. Here, the authors do provide an explicit rate of convergence based on the aforementioned data. Such estimate is moreover independent of $N$, which makes it useful to be used in the case where $N\rightarrow +\infty$ in both the classical mean-field and graphon settings.


The first main result of this article focuses on first-order systems \eqref{e-ODE}. We show that either PE for all $M_{ij}$ or ISC is sufficient to ensure consensus. Moreover, the rate of (exponential) consensus can be explicitly computed.
\begin{theorem}\label{t-consensus}
Let $\{x_i(t)\}_{i=1}^N$ be a solution of \eqref{e-ODE} with initial data $\{\bar x_i\}_{i=1}^N$.

Assume the following conditions:
\begin{description}
\item[(H1) \label{hyp:Lip}] The function $\phi(\cdot):[0,+\infty)\to \mathbb{R}$ is Lipschitz continuous.

\item[(H2) \label{hyp:M}] All weights $M_{ij}:[0,+\infty)\to [0,1]$ are $\mathscr{L}^1$-measurable.

\item[(H3) \label{hyp:phimin}] It holds 
$\phi(r)>0\mbox{~~for all~}r\in[0,+\infty)$.

\end{description}

Fix  $T,\mu>0$ and assume that one of the following cases holds: either all $M_{ij}$ satisfy \eqref{e-PE}, or the family $\{M_{ij}\}_{i,j=1,\ldots,N}$ satisfies \eqref{e-ISC}. Then, there exists $\gamma>0$, which value is explicitly given in Proposition \ref{p-1st} below, such that the following estimate holds:
\begin{eqnarray}\label{e-stimaexp}
&&\max_{i,j}|x_i(nT)-x_j(nT)|\leq \\
&& (1-\gamma)^n \max_{i,j}|x_i(0)-x_j(0)|.\nonumber
\end{eqnarray}
As a consequence, consensus holds too: there exists $x^*\in \R^d$ such that
\begin{eqnarray*}
    \lim_{t \rightarrow + \infty} x_i(t)=x^* \quad \forall i \in \{1,\ldots,N\}.
\end{eqnarray*}

\end{theorem}
The most interesting aspects of this result are the following: first, convergence holds for arbitrarily small $\mu>0$ and arbitrarily large $T>0$. This is strongly related to the fact that cooperative systems are dissipative, in the sense of \cite{CPRT}: even with no interaction, the maximal distance between agents does not increase. Hence, any push towards convergence is not lost forward in time.

The second main result focuses on second-order systems.
\begin{theorem}\label{t-flocking}
Let $\{x_i(t),v_i(t)\}_{i=1}^N$ be a solution of \eqref{e-ODE-2nd} with initial data $\{\bar x_i,\bar v_i\}_{i=1}^N$ and $\lambda_i=1$ for all $i\in \{1,\ldots,N\}$.

Assume that conditions (H1)-(H2)-(H3) of Theorem \ref{t-consensus} and PE holds for all $M_{ij}$. Let $\phi(x)$ be decreasing and satisfy $$\int_0^{+\infty}\phi(r)\,dr=+\infty,$$ i.e. $\phi$ is not integrable.

Then, flocking occurs: there exists $v^*\in\R^d$ such that $\lim_{t\to+\infty} v_i(t)=v^*$ and the trajectories $x_i(t)$ are bounded for all $i \in \{1,\ldots,N\}$.
\end{theorem}
Also in this case, the key interest of this result is the following, at least for the case of $\lambda_i=1$: the condition of non-integrability of $\phi$, that is sufficient to ensure flocking of the original system \eqref{e-ODE-2nd-noM}, also ensures flocking of \eqref{e-ODE-2nd} under either PE for all $M_{ij}$ or ISC condition. Proposition \ref{p-2nd-stima} below provides a more precise estimate about the rate of convergence for the velocity variabile: clearly, this rate is lower in the PE or ISC case than for the original system.\\

The structure of the article is the following. In Section \ref{section: models}, we present some models for systems of the form \eqref{e-ODE-noM} or \eqref{e-ODE}. We also provide some general properties of the solutions. In Section \ref{s-R}, we prove some properties for the dynamics on the real line, that are essential for proving results. The detailed statements and proofs are presented in Section \ref{s-proofs}. We present conclusions and future directions in Section \ref{s-conclusions}.

\section{Models of opinion formation}\label{section: models}

In this section, we describe two important models for opinion formation of the form \eqref{e-ODE-noM}. The same ideas can then be translated to second-order systems of the form \eqref{e-ODE-2nd-noM}.

In the classical case, the function $\phi_{ij}(t)$ is symmetric and $\lambda_i = \lambda$ where $\lambda > 0$ is fixed \cite{HK,degroot1974reaching}. In this setting, the average value is preserved and the system is cooperative. If the influence function satisfies condition {\bf \nameref{hyp:phimin}}, then all $x_i$ converge to the average value. However, such a setting has some scaling problems for large number of agents. Indeed, large groups of agents may have strong impact on small groups, even though they are very far from each other. For this reason, a different scaling has been proposed in \cite{motsch2011new}:
\begin{eqnarray*}
    \lambda_i = \frac{N}{\sum_{j=1}^N \phi_{ij}(t)}.
\end{eqnarray*}
This rescaling introduces asymmetry in the dynamics, thus average is not preserved. Yet, also in this setting one reaches consensus under condition {\bf \nameref{hyp:phimin}}.

We treat both cases in a unitary way, from now on. For this reason, we define
\begin{eqnarray}\label{defof:phimin}
    \phi_{\min} := \min_{r \in [0,\max_{i,j \in \{1,\ldots N\}}|\bar x_i - \bar x_j|]}\phi(r)\\
     \phi_{\max} := \max_{r \in [0,\max_{i,j \in \{1,\ldots N\}}|\bar x_i - \bar x_j|]}\phi(r),\\
\label{defof:lambda_i}
    \lambda_i :=
    \begin{cases}
        1 \quad &\text{for fixed scaling}\\
        \frac{N}{\sum_{j=1}^N \phi_{ij}} \quad &\text{for normalized scaling,}
    \end{cases}\\
\label{defof:kmin}
     K_{\min} :=
    \begin{cases}
        \phi_{\min} \quad &\text{for fixed scaling}\\
        \frac{\phi_{\min}}{\phi_{\max}} \quad &\text{for normalized scaling.}
    \end{cases}\\
\label{defof:kmax}
     K_{\max} :=
    \begin{cases}
        \phi_{\max} \quad &\text{for fixed scaling}\\
        \frac{\phi_{\max}}{\phi_{\min}} \quad &\text{for normalized scaling,}
    \end{cases}
\end{eqnarray}
We use the following inequalities later:
\begin{eqnarray}\label{sum of kernels less than eta max}
    \frac{K_{\min}}{N}\sum_{j=1}^N M_{ij} \leq \frac{\lambda_i}{N}\sum_{j=1}^N M_{ij}\phi_{ij} \leq K_{\max}
\end{eqnarray}
for all $i \in \{1,\ldots, N\}$. They are direct consequences  of the definitions above and the condition $M_{ij} \in[0, 1]$. 

We also later use the following lemma, which proof is a direct computation.
\begin{lemma}\label{l-reverse}
    Given $\{x_1(t), \ldots, x_N(t)\}$ solution of \eqref{e-ODE}, then $\{-x_1(t), \ldots, -x_N(t)\}$ is a solution of \eqref{e-ODE} too.
    \end{lemma}
    The interest of this lemma is that it permits to reverse several results, e.g. properties about the maximum becoming properties about the minimum. 

\subsection{General properties}

In this section, we prove  properties of solutions of \eqref{e-ODE}. We first show that the diameter satisfies a dissipative property. Before doing so, we provide a useful lemma.

\begin{lemma}\label{lemma:scalarproductinequality}
     Let $i,j\in \{1,\ldots,N\}$ be a pair of indices such that
    \begin{eqnarray*}
        \max_{k,l \in \{1,\ldots,N\}}|x_k - x_l|=|x_i - x_j|.
    \end{eqnarray*}
It then holds
\begin{eqnarray*}
&&    \max_{k \in \{1,\ldots,N\}}\langle x_k, x_i - x_j\rangle = \langle x_i, x_i - x_j\rangle \\
 &&   \min_{k \in \{1,\ldots,N\}}\langle x_k, x_i - x_j\rangle = \langle x_j, x_i - x_j\rangle .
\end{eqnarray*}
\end{lemma}
\begin{proof}
    The proof is entirely similar to the proof \cite[Lemma 3.4]{bonnet2022consensus} in the context of graphons.
\end{proof}
\begin{proposition}\label{prop: contract of support}
        
    The function 
 \begin{eqnarray*}
      \gamma_{\max}(t):= \max_{i \in \{1,\ldots,N\}}\left|x_i(t)\right| 
 \end{eqnarray*}
 is non-increasing

 Similarly, the diameter 
 \begin{eqnarray*}
     \D(t):= \max_{i,j \in \{1,\ldots,N\}}\left|x_i(t)-x_j(t)\right| 
 \end{eqnarray*}
 is non-increasing
 
Similarly, on the real line, the function 
 \begin{eqnarray*}
     \gamma_{\min}(t):= \min_{i \in \{1,\ldots,N\}}x_i(t)
 \end{eqnarray*}
 is non-decreasing.
\end{proposition}
\begin{proof}
The functions $\gamma_{\max}$ and $\D$ are Lipschitz, since they are the pointwise maxima of a finite family of Lipschitz continuous functions. By Rademacher's theorem, they are differentiable almost everywhere. By Danskin's theorem (see \cite{danskin2012theory}) it  holds
  \begin{eqnarray*}
       \frac{1}{2}\frac{d}{dt}\gamma_{\max}^2(t) = \max_{i \in \Pi_1(t)}\left\langle \frac{d}{dt}x_i(t),x_i(t)\right\rangle
   \end{eqnarray*}
 where $\Pi_1(t) \in \{1,\ldots,N\}$ represents the nonempty subset of indices for which the maximum $\gamma_{\max}(\cdot)$ is reached. Fix an arbitrary $p\in \Pi_1(t)$ and observe that for all $j \in \{1,\ldots,N\}$ it  holds
\begin{eqnarray*}
    \left\langle x_{p}-x_j, x_{p}\right\rangle \geq 0,
\end{eqnarray*}
 which implies that for all $t \geq 0$ it holds
 \begin{eqnarray*}
     &\left\langle \frac{d}{dt}x_p(t),x_p(t)\right\rangle \\
     &= \frac{\lambda_p}{N}\sum_{j=1}^{N} M_{pj}(t)\phi_{pj}(t) \left\langle x_j(t) - x_p(t), x_p(t) \right\rangle \leq  0.
 \end{eqnarray*}
 Since this estimate holds for any $p \in \Pi_1(t)$, we have
\begin{eqnarray*}
    \frac{d}{dt}\gamma_{\max}(t) \leq 0 \quad \forall t \geq 0,
\end{eqnarray*}
i.e. the function $\gamma_{\max}$ is non-increasing. 

The statement for the size of the support is recovered as follows. Again, by using Danskin's theorem it holds
   \begin{eqnarray*}
       \frac{1}{2}\frac{d}{dt}\D^2(t) = \max_{i,j \in \Pi_2(t)}\left\langle \frac{d}{dt}(x_i(t) - x_j(t)), x_i(t) - x_j(t) \right\rangle
   \end{eqnarray*}
 where $\Pi_2(t) \in \{1,\ldots,N\} \times \{1,\ldots,N\}$ represents the nonempty subset of pairs of indices for which the maximum $\mathcal D(\cdot)$ is reached.  Fix arbitrary $p,q \in \Pi_2(t)$. For easier notation, from now on we hide the dependence on time. Notice that for the case of normalized scaling it holds
\begin{eqnarray*}
        &\left\langle\frac{d}{dt}(x_{p} - x_{q}),x_{p} - x_{q}\right \rangle \\
        &= -\frac{1}{\sum_{k=1}^N \phi_{p k}}\sum_{j=1}^{N} M_{p j} \phi_{p j}\left\langle x_{p}-x_j, x_{p} - x_{q}\right\rangle \\
   & - \frac{1}{\sum_{k=1}^N \phi_{q k}}\sum_{j=1}^{N} M_{q j} \phi_{q j}\left\langle x_j - x_{q},x_{p} - x_{q}\right\rangle .
        \end{eqnarray*}
By Lemma \ref{lemma:scalarproductinequality}, for all $j \in \{1,\ldots,N\}$ it holds
\begin{eqnarray*}
    \left\langle x_{p}-x_j, x_{p} - x_{q}\right\rangle \geq 0 \quad \text{and} \quad \left\langle x_j - x_{q},x_{p} - x_{q}\right\rangle \geq 0
\end{eqnarray*}
and therefore, by \eqref{defof:kmin}, it holds
\begin{eqnarray*}
         &\left\langle\frac{d}{dt}(x_{p} - x_{q}),x_{p} - x_{q}\right \rangle\\
        &\leq -\frac{K_{\min}}{N} \left(\sum_{j=1}^{N} M_{p j} \left\langle x_{p}-x_j, x_{p} - x_{q}\right\rangle \right.\\
        &\quad \left.+ \sum_{j=1}^{N} M_{q j}\left\langle x_j - x_{q},x_{p} - x_{q}\right\rangle \right)\leq 0.
\end{eqnarray*}

The statement for the minimum on the real line can be recovered by Lemma \ref{l-reverse}.
\end{proof}

\section{Dynamics on the real line} \label{s-R}

In this section, we study the case $d=1$, i.e. in which the configuration space is the real line $\R$. In this setting, ordering is clearly a major advantage.

We introduce two useful operators, that are $\psi_L(\alpha,z,\tau)$ and $\psi_R(\beta,z,\tau)$. The idea for $\psi_L$ is that, given an initial configuration in which the position of all agents has a value larger than $\alpha$, the quantity $\psi_L(\alpha,z,\tau)$ represents a lower barrier for the position at time $t=\tau$ of a particle starting at $z$ at time $t=0$: if the barrier is crossed to the right, then the trajectory cannot go back to its left side. A reverse reasoning holds for $\psi_R(\beta,z,\tau)$: if the barrier is crossed to the left, then the trajectory cannot go back to its right side.
\begin{lemma}\label{l-psi}
Let $\alpha, \beta,z\in\R$, $T> 0$ and $\tau \in [0,T]$. Define 
\begin{eqnarray}
\label{e-psi}    
\psi_L(\alpha,z,\tau)&:=&\alpha+e^{-K_{\max}\tau}(z-\alpha),\\
\psi_R(\beta,z,\tau)&:=&\beta-e^{-K_{\max}\tau}(\beta-z)
\end{eqnarray}
where $K_{\max}$ is defined by \eqref{defof:kmax}.
Let $$x(t) := \{x_1(t),\ldots,x_N(t)\}$$  be a solution of \eqref{e-ODE} on $\R$ that satisfies $x_j(\theta)\geq \alpha$ for all $j\in\{1,\ldots,N\}$ for some $\theta \geq 0$. If there exists an index $i$ and $\tau^*\in [0,T]$ such that $x_i(\theta+\tau^*)\geq \psi_L(\alpha,z,\tau^*)$, then
\begin{equation}
\label{e-stimapsiL}  x_i(\theta+\tau)\geq \psi_L(\alpha,z,\tau)~~~\mbox{~for all }\tau\in[\tau^*,T].
\end{equation}
Similarly, let $x(t) := \{x_1(t),\ldots,x_N(t)\} $  be a solution of \eqref{e-ODE} on $\R$ that satisfies $x_j(\theta)\leq \beta$ for all $j\in\{1,\ldots,N\}$ for some $\theta \geq 0$. If there exists an index $i$ and $\tau^*\in [0,T]$ such that $x_i(\theta+\tau^*)\leq \psi_R(\beta,z,\tau^*)$, then
\begin{equation}
\label{e-stimapsiR}  x_i(\theta+\tau)\leq \psi_R(\beta,z,\tau)~~~\mbox{~for all }\tau\in[\tau^*,T].
\end{equation}

\end{lemma}
\begin{proof} We only prove the first estimate, the second one being completely equivalent thanks to Lemma \ref{l-reverse}. Define $\tau:=t-\theta$ and for $t > \theta$ compute 
    \begin{eqnarray*}
        &&\partial_t( x_i(t)-\psi_L(\alpha,z,t-\theta))=\\
        &&\frac{\lambda_i}{N}\sum_{x_j(t)\leq \psi_L(\alpha,z,\tau)} M_{ij}(t)\phi_{ij}(t)(x_j(t) - x_i(t))+\\
        &&\frac{\lambda_i}{N}\sum_{x_j(t)> \psi_L(\alpha,z,\tau)} M_{ij}(t)\phi_{ij}(t)(x_j(t) - x_i(t))\\
        &&+ K_{\max}e^{-K_{\max}\tau}(z-\alpha).
        \end{eqnarray*}
        In the first term,  write 
        \begin{eqnarray*}
        && x_j(t) - x_i(t)\geq\\
        && (\alpha -\psi_L(\alpha,z,\tau))+(\psi_L(\alpha,z,\tau)-x_i(t))\\
        && = -e^{-K_{\max}\tau}(z-\alpha)+(\psi_L(\alpha,z,\tau)-x_i(t)).
        \end{eqnarray*}
    
        Here we used Proposition \ref{prop: contract of support}, since $$x_j(t)\geq \gamma_{\min}(t)\geq \gamma_{\min}(\theta)\geq \alpha.$$
        Using \eqref{sum of kernels less than eta max}, we then estimate 
        \begin{eqnarray*}
        &&\partial_t( x_i(t)-\psi_L(\alpha,z,t-\theta))\geq\\
        &&\frac{\lambda_i}{N}\sum_{x_j(t)\leq \psi_L(\alpha,z,\tau)} M_{ij}(t)\phi_{ij}(t) (-e^{-K_{\max}\tau}(z-\alpha))+\\
        &&\frac{\lambda_i}{N}\sum_{x_j(t)\leq \psi_L(\alpha,z,\tau)} M_{ij}(t)\phi_{ij}(t)
        (\psi_L(\alpha,z,\tau)-x_i(t)))+\\
        &&\frac{\lambda_i}{N}\sum_{x_j(t)> \psi_L(\alpha,z,\tau)} M_{ij}(t)\phi_{ij}(t)(\psi_L(\alpha,z,\tau) - x_i(t))\\
    &&+K_{\max}e^{-K_{\max}\tau}(z-\alpha)\geq K_{\max} (-e^{-K_{\max}\tau}(z-\alpha))\\
      && +\frac{\lambda_i}{N}\sum_{j=1}^N M_{ij}(t)\phi_{ij}(t)
        (\psi_L(\alpha,z,\tau)-x_i(t)))+\\
    &&K_{\max}e^{-K_{\max}\tau}(z-\alpha)=\\
        &&a(t)(x_i(t)-\psi_L(\alpha,z,\tau)),
                \end{eqnarray*}
                where $$a(t):=-\frac{\lambda_i}{N}\sum_{j=1}^{N}M_{ij}(t)\phi_{ij}(t).$$
              
Recall that $x_i(\theta+\tau^*)-\psi_L(\alpha,z,\tau^*)\geq 0$. We now apply Gr\"onwall's inequality on $x_i(t)-\psi_L(\alpha,z,t-\theta)$ with $t\in[\theta+\tau^*,\theta+T]$. It ensures
\begin{eqnarray*}
&&x_i(\theta+\tau)-\psi_L(\alpha,z,\tau)\geq \\
&&e^{\int_{\tau^*}^\tau a(t)\,dt}\left(x_i(\theta+\tau^*)-\psi_L(\tau^*)\right)\geq 0.
\end{eqnarray*}
This proves \eqref{e-stimapsiL}.
\end{proof}

We are now ready to prove a key result: it provides a quantitative estimate showing that it is impossible to have trajectories staying at both extremes of the domain for a whole time interval $[0,T]$.

\begin{lemma} \label{l-2alternative} Let $x(t) := \{x_1(t),\ldots,x_N(t)\} $  be a solution of \eqref{e-ODE} on $\R$. Define
\begin{eqnarray*}
      \alpha := \min_{i \in \{1,\ldots,N\}} x_i(0), \quad  \beta := \max_{i \in \{1,\ldots,N\}} x_i(0).
  \end{eqnarray*}

\begin{itemize}
\item If PE holds for all $M_{ij}$, define:
\begin{equation}\label{e-gammaPE}
\tilde \gamma:=\frac{\mu K_{min}}{N(1+ K_{max}T)+2\mu K_{min}}
\end{equation}
and 
\begin{equation}\label{e-gammaprimePE}
\gamma':=\tilde \gamma(\beta-\alpha).
\end{equation}
\item If ISC holds, define
\begin{equation}\label{e-gammaISC}
\tilde \gamma:=\frac{
\mu K_{min}}{2(
N(1+K_{max}T)+\mu K_{min})}
\end{equation}
and 
\begin{equation}\label{e-gammaprimeISC}
\gamma':=\exp(-K_{max}T)\tilde \gamma(\beta-\alpha).
\end{equation}
\end{itemize}

  If there exists an index $i\in\{1,\ldots,N\}$ satisfying $x_i(t)\leq \psi_L(\alpha,\alpha+\gamma',t)$ for all $t\in[0,T]$, then there exists no index $j\in \{1,\ldots,N\}$ satisfying $x_j(t)\geq \psi_R(\beta,\beta-\gamma',t)$ for all $t\in [0,T]$.
\end{lemma}
\begin{proof} By contradiction, let $i,j$ be such that it both holds $x_i(t)\leq \psi_L(\alpha,\alpha+\gamma',t)$ and $x_j(t)\geq \psi_R(\beta,\beta-\gamma',t)$ for all $t\in[0,T]$. Since $\psi_L$ is decreasing with respect to $t$ and $\psi_R$ is increasing, this implies $x_i(t)\leq\alpha+\gamma'$ and $x_j(t)\geq \beta-\gamma'$, for all $\tau\in[0,T]$.

We now study the two cases separately:
\begin{itemize}
\item Let \eqref{e-PE} hold for all $M_{ij}$. We first estimate $\dot x_i$: it holds
\begin{eqnarray*}
\dot x_i(t)&=&\frac{\lambda_i}{N}\sum_{x_l(t)\leq x_i(t)} M_{il}(t)\phi_{il}(t)(x_l(t) - x_i(t))+\\
        &&\frac{\lambda_i}{N}\sum_{x_l(t)> x_i(t)} M_{il}(t)\phi_{il}(t)(x_l(t) - x_i(t)).
        \end{eqnarray*}
        The first summation contains no more than $N-1$ terms; for each of them, we observe that $x_l(t)\leq x_i(t)\leq \alpha+\gamma'$, hence $x_l(t) - x_i(t)\geq -\gamma'$.
        
        Each term in the second summation satisfies $M_{il}(t)\phi_{il}(t)(x_l(t) - x_i(t))\geq 0$. 
        Moreover, observe that 
        since
        $\tilde \gamma<\frac12$
        we have
        $ \beta-\alpha-2\gamma'\geq 0$.
        Hence,
        we deduce that
        the second summation in particular contains $x_j$, for which it holds $x_j(t)-x_i(t)\geq \beta-\alpha-2\gamma'$. As a result, it holds
        \begin{eqnarray*}
\dot x_i(t)&\geq &-\frac{N-1}{N}K_{max}\gamma'+0+\\
&&\frac{\lambda_i}{N}M_{ij}(t)\phi_{ij}(t)(\beta-\alpha-2\gamma').
        \end{eqnarray*}

By integrating on the time interval $[0,T]$ and recalling the condition \eqref{e-PE}, it holds
\begin{eqnarray*}       
x_i(T)&>&x_i(0)-K_{max} T\gamma'+\mu \frac{K_{min}}N (\beta-\alpha-2\gamma')\geq\\
&\geq &\alpha+(\beta-\alpha)\left(-K_{max} T\tilde\gamma+\mu \frac{K_{min}}N (1-2\tilde\gamma)\right)\\
&=&\alpha+(\beta-\alpha)\tilde\gamma=\alpha+\gamma'.
\end{eqnarray*}
This raises a contradiction.
\item Let \eqref{e-ISC} hold: for the given $i,j$ indexes, there then exists an index $k$ for which \eqref{e-ISC} holds. It either holds $x_k(0)\geq \frac{\beta+\alpha}2$ or $x_k(0)< \frac{\beta+\alpha}2$. We now assume that the first condition holds. We apply Lemma \ref{l-psi} to $x_k(\tau)$ with $\theta=\tau^*=0$: it holds 
\begin{eqnarray*}
&&x_k(t)\geq \psi_L\left(\alpha,\frac{\beta+\alpha}2,t\right)\geq \psi_L\left(\alpha,\frac{\beta+\alpha}2,T\right)=\\
&&\alpha+\exp(-K_{max}T)\left(\frac{\beta+\alpha}2-\alpha\right)=\\
&&\alpha+\exp(-K_{max}T)\left(\frac{\beta-\alpha}2\right).
\end{eqnarray*}

This implies
\begin{eqnarray}
&&x_k(t)-x_i(t)\geq\nonumber\\
&& \alpha+\exp(-K_{max}T)\left(\frac{\beta-\alpha}2\right)-(\alpha+\gamma')=\nonumber\\
&&\exp(-K_{max}T)\left(\frac12-\tilde\gamma\right)(\beta-\alpha). \label{e-xkxi}
\end{eqnarray}
Here, it is useful to observe that $\tilde\gamma<\frac12$.

We are now ready to estimate $\dot x_i$. Estimates are similar to the previous case. It holds
\begin{eqnarray*}
\dot x_i(t)&=&\frac{\lambda_i}{N}\sum_{x_l(t)\leq x_i(t)} M_{il}(t)\phi_{il}(t)(x_l(t) - x_i(t))+\\
        &&\frac{\lambda_i}{N}\sum_{x_l(t)> x_i(t)} M_{il}(t)\phi_{il}(t)(x_l(t) - x_i(t)).
        \end{eqnarray*}
        For the first summation of no more than $N-1$ terms, we again use $x_l(t) - x_i(t)\geq -\exp(-K_{max}T)\tilde\gamma(\beta-
        \alpha)$. For the second summation, we use $M_{il}(t)\phi_{il}(t)(x_l(t) - x_i(t))\geq 0$. In particular, the second term contains $x_k$, for which \eqref{e-xkxi} holds. As a result, it holds
         \begin{eqnarray*}
&&\dot x_i(t)\geq -\frac{N-1}{N}K_{max}\exp(-K_{max}T)\tilde\gamma(\beta-\alpha)+0+\\
&&\frac{\lambda_i}{N}M_{ik}(t)\phi_{ik}(t)\exp(-K_{max}T)\left(\frac12-\tilde\gamma\right)(\beta-\alpha).
        \end{eqnarray*}
        By integrating on the time interval $[0,T]$ and recalling condition \eqref{e-ISC}, it holds
\begin{eqnarray*}       
x_i(T)&>&x_i(0)-K_{max} T\exp(-K_{max}T)\tilde\gamma(\beta-\alpha)+\\
&&\mu \frac{K_{min}}N \exp(-K_{max}T)\left(\frac12-\tilde\gamma\right)(\beta-\alpha)\geq\\
&& \alpha+(\beta-\alpha)\exp(-K_{max}T)\tilde\gamma=\alpha+\gamma'.
\end{eqnarray*}
This raises a contradiction, as in the PE case.

We now assume that $x_k(0)< \frac{\beta+\alpha}2$. With completely similar computations, by applying Lemma \ref{l-psi} to $x_j$ and using $\psi_R$, one proves that $x_j(T)<\beta-\gamma'$, that raises a similar contradiction.
\end{itemize}
\end{proof}

\section{Proof of main results} \label{s-proofs}
In this section, we prove Theorems \ref{t-consensus} and \ref{t-flocking}. We first prove  Theorem \ref{t-consensus} for the 1-dimensional case, then for a general dimension $d>1$. We finally prove Theorem \ref{t-flocking}, as a consequence of Theorem \ref{t-consensus}.

\subsection{Proof of Theorem \ref{t-consensus} in $\R$}

\label{s-proof1}
In this section, we prove Theorem \ref{t-consensus} when the configuration space is $\R$. We prove it by stating and proving the following proposition, in which the rate of convergence is explicitly computed.

\begin{proposition}\label{p-1st} Let $x(t) := \{x_1(t),\ldots,x_N(t)\} $  be a solution of \eqref{e-ODE} on $\R$. Let the hypotheses of Theorem \ref{t-consensus} hold, with given $T,\mu>0$. Let $K_{max}$ be defined by \eqref{defof:kmax} and $\tilde \gamma$ as in Lemma \ref{l-2alternative}. It then holds 
\begin{eqnarray}\label{e-stimaexp2}
&&\max_{i,j}|x_i(nT)-x_j(nT)|\leq\\
&& (1-\exp(-\eta K_{max}T)\tilde \gamma)^n \max_{i,j}|x_i(0)-x_j(0)|.\nonumber
\end{eqnarray}
with $\eta=1$ for PE and $\eta=2$ for ISC.
\end{proposition}
\begin{proof} Let $\alpha:=\min_l x_l(0)$ and $\beta:=\max_l x_l(0)$, so that the diameter is $\D(0)=\beta-\alpha$. Choose indexes $i,j$ such that $x_i(T)=\min_{l\in\{1,\ldots,N\}}x_l(T)$ and $x_j(T)=\max_{l\in\{1,\ldots,N\}}x_l(T)$, so that the diameter is $\D(T)=x_j(T)-x_i(T)$. Define $\tilde \gamma$
and $\gamma'$ by either \eqref{e-gammaPE}, \eqref{e-gammaprimePE}, or \eqref{e-gammaISC}, \eqref{e-gammaprimeISC}.
One has two possibilities:





\begin{itemize}
\item It holds $x_i(\tau)\leq \psi_L(\alpha,\alpha + \gamma',\tau)$ for all $\tau\in [0,T]$. By Lemma \ref{l-2alternative}, it holds $x_j(\tau^*)< \psi_R(\beta,\beta-\gamma',\tau^*)$ for some $\tau^*\in[0,T]$. By Lemma \ref{l-psi}, this implies $x_j(T)<\psi_R(\beta,\beta-\gamma',T)=\beta-\exp(-K_{max}T)\gamma'$. Recalling that 
by Proposition~\ref{prop: contract of support} we have $x_i(T)\geq \alpha$,  it follows
\begin{eqnarray*}
\D(T)&<&\beta-\exp(-K_{max}T)\gamma'-\alpha\\
&= &(1-\exp(-\eta K_{max}T)\tilde \gamma)(\beta-\alpha)\\
&=&(1-\exp(-\eta K_{max}T)\tilde \gamma)\D(0).
\end{eqnarray*}
with $\eta=1$ for PE and $\eta=2$ for ISC.

\item it holds $x_i(\tau^*)> \psi_L(\alpha,\alpha + \gamma',\tau^*)$ for some $\tau^*\in [0,T]$. By Lemma \ref{l-psi}, this implies $x_i(T)> \psi_L(\alpha,\alpha + \gamma',T)=\alpha+\exp(-K_{max}T)\gamma'$. Recalling that 
by Proposition~\ref{prop: contract of support} we have $x_j(T)\leq \beta$, it follows 
\begin{eqnarray*}
\D(T)&<&\beta-\exp(-K_{max}T)\gamma'-\alpha\\
&\leq& (1-\exp(-\eta K_{max}T)\tilde \gamma)\D(0).
\end{eqnarray*}
\end{itemize}

In both cases, this proves 
\begin{eqnarray}\label{e-stimaexp-2}
\D(T)<(1-\exp(-\eta K_{max}T)\tilde \gamma)\D(0).
\end{eqnarray}
 By induction, we find \eqref{e-stimaexp2}.
 \end{proof}
 
 The proof of Theorem \ref{t-consensus} in $\R$ is now straightforward, by choosing $\gamma=\exp(-\eta K_{max}T)\tilde \gamma$.

\subsection{Proof of Theorem \ref{t-consensus} in $\R^d$}

In this section, we prove Theorem \ref{t-consensus} on $\R^d$ for any $d>1$. The idea is to show that the dynamics can be projected on a line, and to use the result of Section \ref{s-proof1} on such line.

The key observation is the following. Fix two vectors $x_0,w\in \R^d$ with $|w|=1$ and take a solution $\{x_i(t)\}$ of \eqref{e-ODE} in $\R^d$. Define the projected solution as 
$$y_i(t):=(x_i(t)-x_0)\cdot w.$$

In general, it is clear that $y_i(t)$ is not a solution of \eqref{e-ODE}, since it holds $$\phi_{ij}(t)=\phi(|x_i(t) - x_j(t)|)\neq \phi(|y_i(t) - y_j(t)|).$$

Yet, a careful look to the proof in Section \ref{s-proof1} shows that the key properties ensuring the result are Proposition~\ref{prop: contract of support} and the estimates for the interaction kernels encoded in \eqref{sum of kernels less than eta max}. In higher dimension, the fact that Proposition~\ref{prop: contract of support} holds also implies that $\phi_{\min},\phi_{\max}>0$ computed for the $x_i$ variables are suitable as bounds for the interaction of the $y_i$ too. As a consequence, Theorem \ref{t-consensus} also holds for the variables $y_i$, for any fixed $x_0,w\in \R^d$, i.e. 
\begin{eqnarray}\label{e-maxY}
&&\max_{ij}|y_i(T)-y_j(T)|\leq(1-\gamma)\max_{ij}|y_i(0)-y_j(0)|.
\end{eqnarray}

Choose now indexes $I,J$ realizing $|x_I(T)-x_J(T)|=\max_{ij}|x_i(T)-x_j(T)|$. If the maximizer is zero, then one already has consensus, since $x_i(T)=x_j(T)$ for all $i,j\in\{1,\ldots,N\}$. Otherwise, choose $x_0=x_{J}(T)$ and $w=\frac{x_{I}(T)-x_{J}(T)}{|x_{I}(T)-x_{J}(T)|}$. The corresponding $y_i$ are given by
$$y_i(t)=(x_i(t)-x_J(T))\cdot \frac{x_{I}(T)-x_{J}(T)}{|x_{I}(T)-x_{J}(T)|},$$
then \eqref{e-maxY}  reads as
\begin{eqnarray*}
&&\max_{ij} ((x_i(T)-x_j(T))\cdot (x_I(T)-x_J(T)))\leq\\
&&(1-\gamma)\cdot \max_{ij}((x_i(0)-x_j(0))\cdot (x_I(T)-x_J(T))).
\end{eqnarray*}
In the left hand side, it is easy to observe that the maximizer is given by $i=I,j=J$, by construction. In the right hand side, it is sufficient to observe that the scalar product is smaller than the product of norms. We then have
\begin{eqnarray*}
&&((x_I(T)-x_J(T))\cdot (x_I(T)-x_J(T)))=\\
&&|x_I(T)-x_J(T)|^2\leq (1-\gamma)\cdot\\
&&\max_{ij}|x_i(0)-x_j(0)|\cdot |x_I(T)-x_J(T)|.
\end{eqnarray*}
Since we assume $x_I(T)\neq x_J(T)$, it holds 
\begin{eqnarray}
&&\D(T)=|x_I(T)-x_J(T)|\leq (1-\gamma)\max_{ij}|x_i(0)-x_j(0)|\nonumber\\
&&=(1-\gamma)\D(0).\label{e-expRd}
\end{eqnarray}
Remark that this formula does not depend on the choice of $x_0,w$. By induction, it holds \eqref{e-stimaexp}.

\subsection{Proof of Theorem \ref{t-flocking}}

In this section, we prove Theorem \ref{t-flocking}. Also in this case, we prove it by providing an explicit rate of convergence.

We first prove a result about flocking: it is equivalent to boundedness of the $x_i$ variables. This was already proved in the case of solutions of \eqref{e-ODE-2nd-noM}, i.e. with constant communication. See \cite{ha-liu}.

\begin{proposition} \label{p-DXbound}
Let $(x(t),v(t))$  be a solution of \eqref{e-ODE-2nd}. Assume that conditions (H1)-(H2)-(H3) of Theorem \ref{t-consensus} hold. Also assume that either PE for all $M_{ij}$ or ISC holds. Define 
\begin{eqnarray}
&&\D_X(t):=\max_{ij}|x_i(t)-x_j(t)|,\label{e-DX}\\
&&\D_V(t):=\max_{ij}|v_i(t)-v_j(t)|.\label{e-DV}
\end{eqnarray}
Then, flocking (defined in Theorem \ref{t-flocking}) is equivalent to the following condition:
\begin{eqnarray}
\label{e-cond1}
&&\lim_{t\to+\infty}\D_V(t)= 0,{~and~}\\
&&
 \D_X(t) \mbox{~bounded for $t\in[0,+\infty)$.}\label{e-DXbounded}
\end{eqnarray}
Moreover, flocking is equivalent to the single condition \eqref{e-DXbounded}.
\end{proposition}
\begin{proof}
We first prove the first equivalence. It is clear that flocking implies \eqref{e-cond1}-\eqref{e-DXbounded}. On the other side, observe that the dynamics of $v_i$ in \eqref{e-ODE-2nd} is of cooperative form, hence the $v_i$ are bounded due to contractivity of the support (Proposition \ref{prop: contract of support}). Then, there exists a subsequence of times $t_k\to+\infty$ such that each $v_i(t_k)$ converges to some $v_i^*$. Condition $\lim_{t\to+\infty}\D_V(t)= 0$ ensures that these $v_i^*$ are all identical. Again by contractivity of the support
and because of~\eqref{e-cond1}, it holds that for all $i \in \{1,\ldots,N\}$ the functions  $v_i(t)$ converge to a common value.

For the second equivalence, 
observe that \eqref{e-DXbounded} implies 
$$\phi(|x_i(t)-x_j(t)|)\geq\min_{r\in[0,\sup(\D_X(t))]}\phi(r)>0$$ for all $t\in[0,+\infty)$,
and for all $i,j$. This implies that the dynamics of $v_i$ in \eqref{e-ODE-2nd} is of cooperative form with a lower bound on the strength of the interaction. One can then apply Theorem \ref{t-consensus} to the $v_i$ variables and prove $\lim_{t\to+\infty}\D_V(t)= 0$.
\end{proof}

We now prove a first estimate about the rate of convergence of the diameter $\D_V$ of the velocity variables.

\begin{proposition} \label{p-2nd-stima} Let $(x(t),v(t))$  be a solution of \eqref{e-ODE-2nd} and $\D_X,\D_V$ defined by \eqref{e-DX}-\eqref{e-DV}. Assume conditions (H1)-(H2)-(H3) of Theorem \ref{t-consensus} hold. Also assume that either PE for all $M_{ij}$ or ISC holds. Finally, assume that $\phi$ is decreasing. 

Define $p:=\phi(0)=\phi_{max}$ and the function $$f(y):=\frac{\exp(-\theta_1 pT) \mu y}{\theta_2+2\mu y}$$ with the following choice of parameters $\theta_1,\theta_2$:
\begin{itemize}
\item for the PE condition with fixed scaling
$$\theta_1=1,\qquad \theta_2=N+NTp;$$
\item for the PE condition with normalized scaling
$$\theta_1=1/y,\qquad \theta_2=Np+NTp^2/y;$$

\item for the ISC condition with fixed scaling
$$\theta_1=2,\qquad \theta_2=2N\exp(pT)(1+Tp);$$

\item for the ISC condition with normalized scaling
$$\theta_1=2/y,\qquad \theta_2=2N\exp(pT/y)(1+Tp/y);$$

\end{itemize}

It then holds
\begin{eqnarray}\label{e-DVnT}
&&\D_V(nT)\leq\\
&& \D_V(0) -\frac{1}{T}\int_{\D_X(0)+T\D_V(0)}^{\D_X(nT)+T\D_V(nT)} f(\phi(x))\,dx.\nonumber
\end{eqnarray}
\end{proposition}
\begin{proof} We first prove \eqref{e-DVnT} for $n=1$. First recall that $\D_V(t)$ is a decreasing function. Then, by integrating \eqref{e-ODE-2nd}, it holds 
$$\D_X(t)\in [\D_X(0)-t\D_V(0),D_X(0)+t\D_V(0)].$$

Recall the definition of $K_{min},K_{max}$ as functions of $\phi_{min}$ given in \eqref{defof:kmin}-\eqref{defof:kmax}. They are increasing as a function of $\phi_{min}$. Also recall the definition of $\tilde\gamma$ given in Lemma \ref{l-2alternative}, as a function of $K_{min},K_{max}$, hence of $\phi_{min}$ only (since $\phi_{max}=\phi(0)$ is fixed). Since $\tilde\gamma$ is increasing with respect to $K_{min}$ and decreasing with respect to $K_{max}$, it is decreasing as a function of $\phi_{min}$, (again observing that $\phi_{\max}$ is fixed). Finally, recall the definition of $\exp(-\eta K_{max}T)\tilde \gamma$ as a function of $\phi_{min}$: by composition of the definitions above, one can check that it coincides with $f(\phi_{min})$, where $f$ is defined in the statement in the 4 different cases. Moreover, it is increasing as a function of $\phi_{min}$. Recall that $\phi$ is decreasing, then we have $\phi_{min}\geq \phi\big(\D_X(0)+T\D_V(0)\big)$.
Moreover it holds $|\D_X(T)-\D_X(0)|\leq T\D_V(0)$. Thus, one can apply Proposition \ref{p-1st}, that ensures
\begin{eqnarray}\label{e-DVT}
&&\D_V(T)\leq (1-f(\phi_{min}))\D_V(0),
\end{eqnarray}
and then we derive
\begin{eqnarray*}
&&\D_V(T)-\D_V(0)\leq \\
&&-f(\phi(\D_X(0)+T\D_V(0))) \frac{|\D_X(T)-\D_X(0)|}T.
\end{eqnarray*}
We rewrite it in integral form, by recalling that $f(\phi(r))$ is decreasing
(since $f$ is increasing), strictly positive and $\D_V(T)\leq \D_V(0)$.
\begin{eqnarray*}
&&\D_V(T)-\D_V(0)\leq\\
&& -\frac1T \int_{\D_X(0)+T\D_V(0)}^{\D_X(T)+T\D_V(0)}
f(\phi(\D_X(0)+T\D_V(0)))\,dr\\
&&\leq -\frac1T \int_{\D_X(0)+T\D_V(0)}^{\D_X(T)+T\D_V(T)}
f(\phi(r))\,dr.
\end{eqnarray*}
In the last inequality we are assuming that $\D_X(0)+T\D_V(0)<\D_X(T)+T\D_V(T)$ since, otherwise, the estimate~\eqref{e-DVnT} would be already true for $n=1$.
Thus, the statement is proved for $n=1$. The general estimate \eqref{e-DVnT} is then recovered by induction.
\end{proof}

We can now provide a quantitative result of flocking. As a particular case, this later implies Theorem \ref{t-flocking}.
\begin{proposition} \label{p-2nd} Let $(x(t),v(t))$  be a solution of \eqref{e-ODE-2nd}. Under the same hypotheses and using the same notation of Proposition \ref{p-2nd-stima}, assume that it holds
\begin{equation}\label{e-cond-flock}
\D_V(0)< \frac{1}{T}\int_{\D_X(0)+T\D_V(0)}^{+\infty} f(\phi(x))\,dx.
\end{equation}
Then flocking occurs.
\end{proposition}
\begin{proof} By using Proposition \ref{p-DXbound}, it is sufficient to prove that $\D_X(t)$ is bounded. By contradiction, assume that $\D_X(t)$ is unbounded. By recalling that $\D_X(t)$ is a Lipschitz function of time, this implies that $\D_X(nT)$ is an unbounded sequence too. Then, there exists a subsequence $n_k\to+\infty$ such that $\D_X(n_kT)$ is increasing. Since $\D_V(n_kT)$ is a bounded sequence, there exists a further subsequence (that we do not relabel) for which the sequence $\D_X(n_kT)+T\D_V(n_kT)$ is increasing.

By recalling that $f(\phi(x))$ is a positive function, the following sequence is decreasing:
$$I_k:=\D_V(0)- \frac{1}{T}\int_{\D_X(0)+T\D_V(0)}^{\D_X(n_kT)+T\D_V(n_kT)} f(\phi(x))\,dx.$$

By hypothesis \eqref{e-cond-flock}, the limit satisfies $\lim_{k\to+\infty}I_k<0$. Then, there exists $k\in\mathbb{N}$ such that  $I_K<0$. By \eqref{e-DVnT}, this implies that $\D_V(n_kT)<0$. This contradicts the fact that $\D_V(t)$ is positive, due to the definition in \eqref{e-DV}.
\end{proof}

\begin{corollary} Let $(x(t),v(t))$  be a solution of \eqref{e-ODE-2nd} and $\D_X,\D_V$ defined by \eqref{e-DX}-\eqref{e-DV}. Assume conditions (H1)-(H2)-(H3) of Theorem \ref{t-consensus} hold. Also assume that either PE for all $M_{ij}$ or ISC holds. Finally, assume that $\phi$ is decreasing.  

Then, flocking occurs if one of the following non-integrability conditions holds:
\begin{itemize}
\item for fixed scaling, either in the case of PE or ISC condition, assume that  $$\int_0^{+\infty}\phi(r)\,dr=+\infty,$$ i.e. $\phi$ is not integrable;
\item for normalized scaling,
either in the case of PE or ISC condition, assume that 
$$\int_0^{+\infty} \exp(-\phi(0)T/\phi(r))\phi(r)^2\,dr=+\infty.$$
\end{itemize}

The first case is Theorem \ref{t-flocking}.
\end{corollary}
\begin{proof} For both
cases, our goal is to prove that, for any $\D_X(0),\D_V(0)$, condition \eqref{e-cond-flock} occurs. It is then sufficient to prove that the integral $$I:=\int_{\D_X(0)+T\D_V(0)}^{+\infty} f(\phi(x))\,dx$$ is unbounded. Observe that boundedness of this integral only depends on the behavior of the function $f(\phi(x))$ for $x\to +\infty$. By recalling that $\phi(x)$ is positive and decreasing, it exists $L=\lim_{x\to+\infty} \phi(x)$. We have two possibilities:
\begin{itemize}
\item it holds $L>0$. It then holds $\lim_{x\to+\infty} f(\phi(x))=f(L)>0$, hence $I=+\infty$.
\item it holds $L=0$. We can then approximate $f(y)$ around 
$y=0$
to study boundedness of the integral. We have two
cases:
\begin{itemize}
\item in the cases of fixed scaling (both for the PE and ISC condition), it holds $$f(y)= C y+o(y),$$ in a neighborhood of $y=0^+$. 
\item in the case of normalized scaling (both for the PE and ISC condition), it holds
$$f(y)=C\exp(-2pT/y)y^2+o(\exp(-2pT/y)y^2),$$ 
in a neighborhood of $y=0^+$.
\end{itemize}
In 
each of the above cases, 
the corresponding non-integrability condition ensures that $I=+\infty$.
\end{itemize}
\end{proof}
\begin{remark} For the case of fixed scaling, one observes that condition \eqref{e-cond-flock} is a generalization of the results for flocking under no lack of interactions. They can be recovered by considering $T=\mu$, that  ensures $M_{ij}=1$ a.e.. In this case, Proposition \ref{p-2nd} coincides with \cite[Theorem 3.1]{piccoli2015control}.\\
\end{remark}

\section{Conclusions and future directions}\label{s-conclusions}
In this article, we have proved sufficient conditions for convergence of first-order systems to consensus and of second-order systems to flocking, under communication failures. We introduced quantitative estimates about the minimum level of service, namely the PE and ISC conditions. We proved that, for each of these conditions, consensus and flocking can be achieved under the classical conditions for systems with no communication failures. Yet, the rate of convergence is slower. In both consensus and flocking problems, we provided explicit rates of convergence.

In the future, we aim to find even weaker conditions for communication failures, to understand the (theoretical) minimum level of service ensuring consensus or flocking. Moreover, we aim to find explicit rates of convergence for other similar quantities, such as the algebraic connectivity studied in \cite{bonnet2023consensus}.

\bibliographystyle{ieeetr}
\bibliography{refPE}

\end{document}